\documentclass[11pt]{amsart}

\usepackage[english]{babel}
\usepackage[T1]{fontenc}
\usepackage{amsmath,amssymb}
\usepackage[all]{xy}
\usepackage{graphicx}
\usepackage{fancyhdr}
\usepackage{color}

\newtheorem{definition}{Definition }[section]

\newtheorem{lemma}[definition]
{Lemma}
\newtheorem{theorem}[definition]
{Theorem }
\newtheorem{ex}[definition]
{Example }

\newtheorem{prop}[definition]{Proposition}

\newtheorem{rmk}[definition]{Remark}

\newcommand{\lgw}{\longrightarrow}

\newcommand{\ovl}{\overline}
\newcommand{\Frac}{\text{Frac}}

\newcommand{\Id}{\text{Id}}
\newcommand{\Ker}{\operatorname{Ker}}

\newcommand{\wdh}{\widehat}

\newcommand{\Ann}{\operatorname{Ann}}
\renewcommand{\l}{\lambda}

\renewcommand{\O}{\mathcal{O}}

\newcommand{\U}{\mathcal U}
\newcommand{\V}{\mathcal V}
\newcommand{\m}{\mathfrak{m}}

\newcommand{\R}{\mathbb{R}}
\newcommand{\K}{\mathbb{K}}

\newcommand{\N}{\mathbb{N}}
\newcommand{\D}{\Delta}
\newcommand{\I}{\text{I}}

\newcommand{\C}{\mathbb{C}}

\renewcommand{\a}{\alpha}

\renewcommand{\phi}{\varphi}
\renewcommand{\d}{\delta}

\renewcommand{\lg}{\langle}
\newcommand{\rg}{\rangle}

\newcommand{\at}{^o}

\newcommand{\sr}[1]%
{\ifmmode{}^\dagger\else${}^\dagger$\fi\ifvmode
\vbox to 0pt{\vss
 \hbox to 0pt{\hskip\hsize\hskip1em
 \vbox{\hsize3cm\raggedright\pretolerance10000
 \noindent #1\hfill}\hss}\vss}\else
 \vadjust{\vbox to0pt{\vss%
 \hbox to 0pt{\hskip\hsize\hskip1em%
 \vbox{\hsize3cm\raggedright\pretolerance10000%
 \noindent #1\hfill}\hss}\vss}}\fi%
}

\begin{document}
\title{Multiparameter perturbation theory of matrices and  linear operators}
\author{Adam Parusi\'nski}
\email{adam.parusinski@unice.fr}
\address{Universit\'e C\^ote d'Azur, Universit\'e Nice Sophia Antipolis, CNRS, LJAD, Parc
    Valrose, 06108 Nice Cedex 02, France}

\author{Guillaume Rond}
\email{guillaume.rond@univ-amu.fr}
\address{Aix Marseille Univ, LASOL, UMI2001, UNAM, Mexico}


\thanks{Research partially supported by ANR project LISA (ANR-17-CE40-0023-03)}

\begin{abstract}
We show that a normal matrix $A$ with coefficients in $\C[[X]]$, $X=(X_1, \ldots, X_n)$, can be diagonalized, provided the discriminant $\Delta_A $ of its characteristic polynomial is a monomial times a unit.  The proof is an adaptation of our proof of the Abhyankar-Jung Theorem.
As a corollary we obtain the singular value decomposition for an arbitrary matrix $A$ with coefficient in $\C[[X]]$ under a similar assumption on $\Delta_{AA^*} $ and $\Delta_{A^*A} $. 

We also show real versions of these results, i.e. for coefficients in $\R[[X]]$, and deduce several results on multiparameter perturbation theory for normal matrices  with real analytic, quasi-analytic, or Nash coefficients.  
\end{abstract}


\maketitle

\noindent


\section{Introduction}   
The classical problem of perturbation theory of linear operators can be stated as follows. Given a family of linear operators or matrices depending on parameters, with what regularity can we parameterize the eigenvalues and the eigenvectors?

This problem was first considered for families depending on one parameter.  For the analytic dependence the classical results are due to Rellich \cite{relich37, relich42, relich69},  and Kato \cite{kato76}.  For instance, by \cite{kato76} the eigenvalues, eigenprojections, and eigennilpotents of a holomorphic curve of ($n \times n$)-matrices are holomorphic in a complement of a discrete set with at most algebraic singularities. By  \cite{relich42} the eigenvalues and eigenvectors of a real analytic curve of Hermitian matrices admit real analytic 
parametrization.  

More recently, the multiparameter case has been considered, first by Kurdyka and Paunescu 
\cite{KP} for real symmetric and antisymmetric matrices depending analytically on real parameters, and then for normal matrices by Rainer \cite{rainer2011},  \cite{rainer2013} depending again on real parameters.  
The main results of \cite{KP}, 
\cite{rainer2011} and  \cite{rainer2013} state that the eigenvalues and eigenspaces depend analytically on the parameters after blowings-up in the parameter space.  Note that for normal matrices this generalizes also the classical one-parameter case (there are no nontrivial blowings-up of one dimensional {}nonsingular space).  
For a review of both classical and more recent results see \cite{rainer2011} and \cite{rainertext}. 
\\

In this paper we show, in Theorem \ref{thm:normaltheorem}, that the families of normal matrices depending on a formal multiparameter can be diagonalized formally under a simple assumption that the discriminant of its characteristic polynomial (or the square-free form of the characteristic polynomial in general) equals a monomial times a unit. 
 Of course, by the resolution of singularities, one can make the discriminant normal crossings by blowings-up and thus recover easily the results of \cite{KP}, \cite{rainer2011}, and  \cite{rainer2013},  see Section \ref{sec:rectilinear}.  

 As a simple corollary of the main result we obtain in Section \ref{sec:SVD} similar results for the singular value decomposition of families of arbitrary, not necessarily normal, 
matrices. Again, by the resolution of singularities, we can make the discriminant of the family normal crossings by blowings-up.  This way we obtain a global version of the singular value decomposition theorem after blowings-up in both the real case and the complex one.

Our choice of the formal dependence on parameters is caused by the method of proof that is purely algebraic, but it implies analogous results for many Henselian subrings of the ring of formal power series, see Section \ref{sec:henselian}, in particular, for the analytic, quasi-analytic, and algebraic power series (i.e. Nash function germs). The assumption that the rings are Henselian can not be dropped, if we want to study the eigenvalues in terms of the coefficients of the matrix, or its characteristic polynomial, we need the Implicit Function Theorem. 

All these results are of local nature. In the last section we give a simple  example of a  global statement {of a family of matrices defined on an open set $U$ that can be diagonalized globally on $U$.  This is true under the assumption that the discriminant of its characteristic polynomial is locally normal crossings at every point of $U$ and that $U$ is simply connected (see Theorem \ref{global_statement})}.  
We do not know a fully satisfactory general global theorem and we would like to state it as an open problem.
\\

Another novelty of this paper is the method of proof.  
Recall that in \cite{KP} the authors  first reparameterize (by blowing-up) the parameter space in order to get the eigenvalues real analytic.  Then they  solve linear equations describing  the eigenspaces corresponding to irreducible factors of the characteristic polynomial.  This requires to resolve the ideal defined by all the minors of the associated matrices.  A similar approach is adapted in  \cite{rainer2011} and  \cite{rainer2013}.  First the eigenvalues are made analytic by blowings-up and then further blowings-up are necessary, for instance to make the coefficients of matrices and their differences normal crossing.

Our approach is different.  We adapt the algorithm of the proof of 
Abhyankar-Jung Theorem of \cite{PR}, and use a version of Hensel's Lemma to handle directly the matrices (and hence implicitly the eigenvalues and eigenspaces at the same time).  This simplifies the proof and avoids unnecessary blowings-up.  We note that we cannot deduce our result directly from the Abhyankar-Jung Theorem.  Indeed, even under the assumption that the discriminant of the characteristic polynomial is a monomial times a unit,  the Abhyankar-Jung Theorem 
implies only that its roots, that is the eigenvalues of the matrix, are fractional power series of the parameters, that is the power series with  positive rational exponents.

In a recent paper, Grandjean \cite{grandjean} shows results similar to these of 
\cite{KP}, \cite{rainer2011} and  \cite{rainer2013} but by a different approach.  Similarly to our strategy, he does 
not treat the eigenvalues first.  Otherwise his approach is quite different.  He 
considers the eigenspaces defined on the complement of the discriminant locus, denoted $D_A$, and constructs an ideal sheaf $\mathcal F_A$ with the following property.  If $\mathcal F_A$ is principal then the eigenspaces 
extend to $D_A$.  The construction of the ideal sheaf $\mathcal F_A$ is quite 
involved, we refer the reader to \cite{grandjean} for details.

 \subsection{Notation and conventions}
 For a commutative ring $R$ and positive integers $p$ and $q$, we denote by $Mat_{p,q}(R)$ the set of matrices with entries in $R$ with $p$ rows and $q$ columns. When $p$ and $q$ are equal to a same integer $d$, we denote this set by $Mat_d(R)$.

 Let $X=(X_1,\ldots, X_n)$ represent an $n$-tuple of  indeterminates. These indeterminates will be replaced by real variables in some cases. We denote by $\K[X]$ (resp. $\K[[X]]$, resp. $\K\{X\}$) the ring of polynomials (resp. formal power series, resp. convergent power series) in $X_1$, \ldots, $X_n$.

 We say that $f \in \C [[X]]$ is \emph{a monomial times unit} if $f = X^\alpha a (X) = X_1^{\alpha_1} \cdots X_n^{\alpha_n} a (X)$ with $a(0)\ne 0$. 

  For a matrix $A=A(X)\in Mat_d (\C [[X]])$, we denote by $A^*$ its adjoint, i.e. if the entries of $A(X)$ are the series
$$a_{i,j}(X)=\sum_{\a\in\N^n} a_{i,j,\a}X^\a$$
then $A^*(X)$ is the matrix whose entries are the $b_{i,j}(X)$ defined by
$$b_{i,j}(X)=\overline{a}_{j,i}(X)=\sum_{\a\in\N^n}\overline a_{j,i,\a}X^\a.$$

A matrix $A\in Mat_d(\C[[X]])$ is called \emph{normal} if $AA^*=A^*A$ and  \emph{unitary} if $AA^*=A^*A=\I_d$. The set of unitary matrices is denoted by 
$U_d(\C[[X]])$.

For a matrix $A\in Mat_d(\C[[X]])$, we denote by  $P_A (Z) = Z^d+ 
c_1 (X) Z^{d-1}+\cdots+c_d(X) $ its characteristic polynomial  and by 
  $\Delta_A \in \C [[X]]$ the first nonzero generalized discriminant of $P_A(Z)$. Let us recall that $\Delta_A$ equals
  $$\sum_{r_1<\cdots<r_l}\prod_{i<j; i,j\in\{r_1,\ldots, r_l\}} (\xi_i-\xi_j)^2$$
  where the $\xi_i$ are the roots of $P_A(Z)$ in an algebraic closure of  $\C(\!(X)\!)$ and $l$ is the number of such distinct roots. Since $\D_A$ is symmetric in the $\xi_i$ it is a polynomial in the $c_k$. Let us notice that
  \begin{equation}\label{disc}\Delta_A=\mu_1\ldots \mu_l \Delta_A'\end{equation}
  where the $\mu_i$ are the multiplicities of the distinct roots of $P_A$ and $\Delta_A'$ is the discriminant of  
  the reduced (i.e. square-free) 
  form $(P_A)_{red}$ of its characteristic polynomial. One can look at \cite[Appendix IV]{W} or \cite[Appendix B]{PP} for more properties of these generalized discriminants 
  or  subdiscriminants), 
  and to \cite{Roy} or \cite{BPRbook} for an effective way of computing them. \\

\section{Reduction of normal matrices} 
\subsection{A version of Hensel's Lemma for normal matrices}
We begin by stating and proving the main technical tool for 
the reduction of normal matrices. This result is a strengthened version of Cohn's version of Hensel's Lemma (see \cite[Lemma 1]{cohn}).

       \begin{lemma}\label{lem:SplitMat}
 Let $A(X)\in Mat_d(\C[[X]])$ be a normal matrix.      Assume that 
$
A(0) =   \begin{pmatrix}
  B\at _1 & 0 \\
  0 & B\at _2 
  \end{pmatrix}, 
$
 with $B\at _i \in Mat_{d_i} (\C)$, $d=d_1+d_2$,  and such that  the characteristic polynomials of $B\at _1$ and 
 $B\at _2$ are coprime.  \\
Then  there is a unitary matrix  $U \in U_d (\C[[X]])$, $U(0) = \I_d$, such that 
   \begin{align}\label{eq:equationforA}
U ^{-1} A U =   \begin{pmatrix}
B_1 & 0 \\
  0 &  B_2 
  \end{pmatrix}, 
  \end{align} 
  and $B_i (0) = B\at _i$, $i=1,2$.  
  \end{lemma} 

\begin{proof}
Consider  
\begin{align*} 
&\Psi =(\Psi_1, \Psi_2, \Psi _3,\Psi_4)   : \\ 
& U_d(\C[[X]]) \times  Mat_{d_1} (\C[[X]]) \times Mat_{d_2} (\C[[X]]) \times 
 Mat_{d_2,d_1} (\C[[X]]) 
 \to 
  Mat_{d} (\C[[X]])  , 
  \end{align*} 
defined by 
  \begin{align}
(U ,  Y_1 , Y_2, Y_3)  \to    U  \begin{pmatrix}
B\at _1 + Y_1 & 0 \\
  Y_3 &  B\at _2 +Y_2 
  \end{pmatrix}  U^{*} 
 =  \begin{pmatrix}
 T_1 & T_4 \\
  T_3 &   T_2 
  \end{pmatrix} .
  \end{align}     
  where $\Psi_i(U,Y_1,Y_2,Y_3)=T_i$, $i=1,2,3,4$.\\
  Recall that a tangent vector at $\I_d$ to $U_d(\C[[X]]) $  is a matrix $\mathbf u$ that is skew-hermitian 
  $\mathbf u = - \mathbf u^*$.  We shall write it as 
    \begin{align}
\mathbf u  =   \begin{pmatrix}
\mathbf z_1  & \mathbf x \\
- \mathbf x^* &  \mathbf z_2
  \end{pmatrix}.  
  \end{align}
The differential of $\Psi$ at $(\I_d,0,0,0)$ on the vector  $(\mathbf u , \mathbf  y_1 , \mathbf y_2, \mathbf y_3)$ 
is given by 
  \begin{align}
& d\Psi_i  (\mathbf u , \mathbf  y_1 , \mathbf y_2, \mathbf y_3) = \mathbf y_i  + \mathbf z_i B\at _i - B\at _i \mathbf z_i, \qquad i=1,2 \\
& d\Psi_3  (\mathbf u , \mathbf  y_1 , \mathbf y_2, \mathbf y_3)=  \mathbf y_3 -  \mathbf x^* B\at _1 + B\at _2  \mathbf x^* , 
\\
& d\Psi_4  (\mathbf u , \mathbf  y_1 , \mathbf y_2, \mathbf y_3) = \mathbf x 
B\at _2 - B\at _1  \mathbf x. 
  \end{align}
 This differential is a linear epimorphism thanks to  Lemma \ref{lem:Cohn}, 
 that we state and prove below, due to  Cohn \cite{cohn}, see also \cite{zurro}.  Therefore, we may apply the Implicit Function Theorem (IFT).  \\
 More precisely, we apply the IFT to the following map of finitely dimensional manifolds  
 \begin{align*} 
&\Psi_{|_M}  : 
M:= U_d(\C) \times  Mat_{d_1} (\C) \times Mat_{d_2} (\C) \times 
 Mat_{d_2,d_1} (\C) 
 \to 
  Mat_{d} (\C)  , 
  \end{align*} 
  that by Lemma \ref{lem:Cohn} is a submersion at $(\I_d,0,0,0)$. Note that the unitary group $U_d(\C)$ is not a complex manifold but only a nonsingular real algebraic variety.  Therefore, it is convenient to work in the Nash real algebraic set-up.  By the Nash IFT, see e.g. Corollary 2.9.8 of \cite{BCR},
   there exist open sets $\U\subset M$, $\V\subset \R^{2d^2}=Mat_d(\C)$, with $(\I_d,0,0,0)\in \U$ and $\Psi(\I_d,0,0,0)=A(0)\in\V$, and local Nash diffeomorphisms
  $$\theta_1:\U'\subset \R^N\lgw \U,\ \ \theta_1(0)=(\I_d,0,0,0)$$
  $$\theta_2:\V\lgw \V'\subset \R^{2d^2},\ \ \theta_2(A(0))=0$$
  such that $\theta_2\circ\Psi_{|_M}\circ\theta_1(t_1,\ldots, t_N)=(t_1,\ldots, t_{2d^2})$. Here $N$ is the dimension of $M$ as a real manifold, i.e. $N=d^2+2d_1^2+2d_2^2+2d_1d_2$. The condition that $\theta_i$ are Nash diffeomorphisms 
   means that their components are given by algebraic power series with real coefficients.\\
  Now we have that $A(X)=A(0)+\ovl A(X)$ where $\ovl A(0)=0$. Therefore $\theta_2(A(X))$ is well defined and
  $$\theta_2(A(0)+\ovl A(X))=(t_1(X),\ldots, t_{2d^2}(X))$$
  where the $t_i(X)$ are real (formal) power series vanishing at 0. Let us choose freely real (formal) power series $t_{2d^2+1}(X)$, \ldots, $t_N(X)$ vanishing at 0. We set
  $$(U(X),Y_1(X),Y_2(X),Y_3(X))=\theta_1(t_1(X),\ldots, t_N(X)).$$ 
 This is well defined since the $t_i(X)$ are power series vanishing at 0. Then we have 
 $$\Psi(U(X),Y_1(X),Y_2(X),Y_3(X))=A(X)$$
 $$\text{and } (U(0),Y_1(0),Y_2(0), Y_3(0))=(\I_d,0,0,0).$$
This means that there are matrices  $B_1=B\at _1+Y_1(X)$, $B_2=B\at_2+Y_2(X)$, $B_3=Y_3(X)$ such that 
   \begin{align}\label{eq:forA}
U ^{-1} A U =   \begin{pmatrix}
B_1 & 0 \\
 B_3 &  B_2 
  \end{pmatrix}. 
  \end{align} 
  The matrix on the right-hand side is normal and block triangular.  Therefore it is block diagonal.  This ends the proof of lemma.
\end{proof}

\begin{rmk}\label{rem_red}
Lemma \ref{lem:SplitMat} remains valid if we replace $\C[[X]]$ by any subring containing the ring of algebraic power series and stable under composition with algebraic power series.
\end{rmk}

\begin{rmk} The matrix $U$ is not unique since $N>2d^2$.
\end{rmk}

\begin{lemma}\label{lem:Cohn}\cite[Lemma 2.3]{cohn}\cite{zurro}
Let $R$ be an unitary commutative ring, $A \in  Mat_{p} (R)$, $B \in Mat_q (R)$, $C \in Mat_{p,q} (R)$, such that  $P_A$ and $P_B$ are coprime, i.e. there exist polynomials $U$ and $V$ such that $UP_A+VP_B=1$. Then there is a  matrix $M\in Mat_{p,q}(R)$ such that $AM - MB = C$.
\end{lemma}

\begin{proof}
By assumption  there exist polynomials $U$ and $V$ such that
$UP_{A}+VP_{B}=1.$
Set $Q=VP_{B}$. Then $Q(A)=\I_p$ and $Q(B)=0$. Let us write $Q(T)=\sum_{i=0}^rq_iT^i$ and
set 
$M=\sum_{i=1}^rq_i\sum_{k=0}^{i-1}A^kCB^{i-k-1}.$
 Then
$$AM-MB=A\sum_{i=1}^rq_i\sum_{k=0}^{i-1}A^kCB^{i-k-1}-\sum_{i=1}^rq_i\sum_{k=0}^{i-1}A^kCB^{i-k-1}B=$$
$$=\sum_{i=0}^rq_iA^i C-C\sum_{i=0}^rq_iB^i=Q(A)C-CQ(B)=C.$$
\end{proof}

\subsection{Complex normal matrices}

\begin{theorem} \label{thm:normaltheorem}
Let $A(X)=(a_{i,j})_{i,j=1,\ldots, d}\in Mat_d (\C[[X]])$ be normal and suppose that  
$\Delta_A = X_1^{\alpha_1} \cdots X_n^{\alpha_n} g (X)$ 
with $g(0)\ne 0$.  Then there is a unitary matrix 
 $U \in U_d (\C[[X]])$ such that 
$$U(X)^{-1} A(X) U(X)= D(X),$$ 
where $D(X)$ is a diagonal matrix with entries in $\C[[X]]$.    

If, moreover, the last nonzero coefficient of  $P_A$ is a monomial times a unit, then the nonzero entries of $D(X)$  are also of the form a monomial times a unit $X^\alpha a (X)$ and their exponents $\alpha \in \N^n$ are well ordered.   
\end{theorem}


\begin{proof}[Proof of Theorem \ref{thm:normaltheorem}] 
We prove Theorem \ref{thm:normaltheorem} by induction on $d$.  Thus we suppose that the theorem holds for matrices of order less than $d$.   
Our proof follows closely the proof of Abhyankar-Jung Theorem  given in \cite{PR}, that is algorithmic and  based on Theorem 1.1 of \cite{PR}.  The analog of this theorem for our set-up is Proposition \ref{lem:MainLemma}.  For its proof we will need 
the following easy generalization of Theorem 1.1 of \cite{PR} 
to the case of matrices with a not necessarily reduced characteristic polynomial.

\begin{prop}\label{lem:PRlemma} 
Let $P (Z) = Z^d+c_2 (X) Z^{d-2}+\cdots+c_d(X) \in \C[[X]][Z]$ and suppose that there is $c_i \not \equiv 0$.  
 If the discriminant $\Delta $ of $(P)_{red}$ equals a monomial times a unit, then 
the ideal $(c_i^{d!/i}(X))_{i=2,\ldots ,d}\subset \C [[X]]$ is principal and generated by a monomial.  
\end{prop}

\begin{proof} 
By the Abhyankar-Jung Theorem, see e.g. \cite{PR}, there is $q\in \N^n$, $q_i\ge 1$ for all $i$, 
such that the roots of $P_{red}$  are in $\C[[X^{1/q}]]$  and moreover  their differences are fractional monomials.  
The set of these roots (without multiplicities) coincides with the set of roots of $P$.  
Then we argue as in the proof of Proposition 4.1 of \cite{PR}.  
\end{proof}

We note that the exponents  make the $c_i^{d!/i}(X)$ for $i=2,\ldots ,d$
 homogeneous of the same degree as functions of the roots of $P$.  In the case of the characteristic polynomial of a matrix, these coefficients will become homogeneous of the same degree in terms of the entries of the matrix.  

Proposition \ref{lem:PRlemma} implies easily its analog for 
normal matrices.

\begin{prop}\label{lem:MainLemma} 
Suppose that  the assumptions of Theorem  \ref{thm:normaltheorem} are satisfied and that, moreover, $A$ is 
nonzero and $Tr (A(X))=0$.  Then the ideal  $(a_{ij})_{i,j=1,\ldots ,d}\subset \C [[X]]$ is principal and generated by a monomial.  
\end{prop}

\begin{proof} 
We denote by $P_A(Z) = Z^d+c_2 (X) Z^{d-2}+\cdots+c_d(X) \in \C[[X]][Z]$ the characteristic polynomial of $A(X)$. Since $Tr(A(X))=0$ we have that $c_1(X)=0$. Since $A(X)$ is nonzero, one of the $c_i$ is nonzero. Therefore, 
  by Proposition \ref{lem:PRlemma} and \eqref{disc},  the ideal $(c_i^{d!/i}(X))_{i=2,\ldots ,d}$ is principal and generated by a monomial.  This is still the case if we divide $A$ by the maximal monomial  that  divides all entries of $A$.  
  Thus we may assume that no monomial (that is not constant) divides $A$.  
If $A(0)=0$ then there is $j$ such that all  the coefficients $c_i(X)$ of $P_A$ are divisible  $X_j$.  Therefore, for normal matrices, by Lemma \ref{lem:mainlemma}, $A_{|X_j=0}=0$, that means that all entries of $A$ are divisible by $X_j$,  a contradiction. Thus $A(0)\ne 0$ that ends the proof.    
\end{proof}

       \begin{lemma} \label{lem:mainlemma}
Let $A(X)\in Mat_d (\C[[X]])$ be normal.   If every coefficient of $P_{A }$ is zero: $c_i(X) = 0$, $i=1, \ldots , d$, then $A= 0$. 
\end{lemma}

\begin{proof}
Induction on the number of variables $n$. The case $n=0$ is obvious since the matrix $A(0)$ is normal.\\
Suppose $c_i(X) =0$ for $i=1, \ldots , d$.  Consider $A_1 = A_{|X_1=0}$.  By the inductive assumption $A_1\equiv 0$, that is every entry of $A$ is divisible by $X_1$.  If $A\ne 0$ then we divide it by the maximal power  
$X_1^{m} $  that divides all coefficients of $A$.  The resulting matrix, that we denote by $\tilde A$, is normal and the coefficients of its characteristic polynomial  
$P_{\tilde A} $ are $\tilde c_i (X)  = X_1^{-i m} c_i(X)=0$.  This is impossible because then $P_{\tilde A_1}=0$ and $\tilde A_1\ne 0$, that contradicts the inductive assumption. 
\end{proof}

Now we can finish the proof of Theorem \ref{thm:normaltheorem}. We suppose that $A$ is nonzero and  make a sequence of reductions simplifying the form of $A(X)$.  First we note that we may assume 
 $Tr (A(X))=0$.  Indeed, we may replace $A(X)$ by $\hat A(X) = A-Tr (A(X)) \Id$.  Then we may apply Proposition 
 \ref{lem:MainLemma}   and hence, after dividing  $A$ by the maximal monomial that divides all entries of $A$, 
 assume that    $A(0)\ne  0$.

Thus suppose  $A(0)\ne 0$ and $Tr (A(X))=0$.  
 Denote by $ P\at(Z)$ the characteristic polynomial  of $A(0)$.  Since $A(0)$ is normal, nonzero, of trace zero, it has at least two distinct eigenvalues.  Therefore, 
 after a unitary change of coordinates, we may assume that $A(0)$ is block diagonal   
   \begin{align}
A(0) =   \begin{pmatrix}
  B_1\at & 0 \\
  0 &  B_2 \at
  \end{pmatrix}, 
  \end{align} 
 with $B_i\at \in Mat_{d_i} (\C)$, $d=d_1+d_2$,  and with the resultant of the characteristic polynomials of $B_1\at$ and 
 $B_2\at$  nonzero.  By Lemma \ref{lem:SplitMat} there is a unitary matrix  $U \in U_d (\C[[X]] )$, $U(0) = \I_d$, such that 
   \begin{align}\label{eq:equationforA}
U ^{-1} A U =   \begin{pmatrix}
B_1 & 0 \\
  0 &  B_2 
  \end{pmatrix}, 
  \end{align} 
  and $B_i (0) = B_i\at$, $i=1,2$. \\
Note that the matrices $B_i$ satisfying the formula \eqref{eq:equationforA} have to be normal since $A$ is normal.  Moreover,  $P_{U ^{-1} A U} = P_A= P_{B_1} P_{B_2}$.  This shows that the discriminants of 
    $(P_{B_1 })_{red}$ and $(P_{B_2 })_{red}$ divide the  $\Delta_A $ and hence we may apply to $B_1$ and $B_2$ the inductive assumption.

 For the last claim we note that the extra assumption implies that each nonzero eigenvalue of $A$ is a monomial times a unit.  Moreover the assumption on the discriminant implies the same for all nonzero 
differences of the eigenvalues.  Therefore by \cite[Lemma 4.7]{BM1}, the exponents of these monomials are well ordered. The proof of 
Theorem \ref{thm:normaltheorem} is now complete.
           \end{proof}
       


\subsection{Real normal matrices}

This is the real counterpart of Theorem \ref{thm:normaltheorem}.

\begin{theorem} \label{thm:realnormaltheorem}
Let $A(X)\in Mat_d (\R[[X]])$ be normal and suppose that  $\Delta_A = X_1^{\alpha_1} \cdots X_n^{\alpha_n} g (X)$ 
with $g(0)\ne 0$.   Then there exists an orthogonal matrix 
 $O \in Mat_d (\R[[X]])$ such that 
\begin{align}\label{eq:bigmatrix}
O(X)^{-1} \cdot A(X) \cdot O(X)=\left[\begin{array}{cccccc}C_1(X) & & & & &   \\
           & \ddots & & & 0 & \\
           & & C_s(X) & & & \\
           & & & \l_{2s+1}(X) & & \\
           & 0 & & & \ddots & \\
           & & & & & \l_d(X) \end{array}\right] ,
\end{align} 
where $s\geq 0$, $\l_{2s+1}(X)$, \ldots, $\l_{d}(X)\in\R[[X]]$ and the $C_i(X)$ are $(2\times 2)$-matrices of the form
\begin{align}\label{blocs}
\left[\begin{array}{cc} a(X) & b(X)\\ -b(X) & a(X)\end{array}\right]
\end{align}
           for some $a(X)$, $b(X)\in\R[[X]]$.
         If $A(X)$ is symmetric we may assume that $s=0$, i.e. $O(X)^{-1} \cdot A(X) \cdot O(X)$ is diagonal.

If, moreover, the last nonzero coefficient of  $P_A$ is a monomial times a unit, then the nonzero entries of $O(X)^{-1} \cdot A(X) \cdot O(X)$ are of the form a monomial times a unit $X^\alpha a (X)$ and their exponents $\alpha \in \N^n$ are well ordered. 
\end{theorem}

\begin{proof}
This corollary follows from Theorem \ref{thm:normaltheorem} by a classical argument.

By Theorem \ref{thm:normaltheorem} there exists an orthonormal basis of eigenvectors   of $A(X)$ in $\C[[X]]^d$ such that 
the corresponding eigenvalues  are 
$$\l_1(X), \overline \l_1 (X), \ldots, \l_s(X), \overline \l_s(X), \l_{2s+1}(X),\ldots, \l_d(X), $$
where $\l_i(X)\in\C[[X]]\backslash \R[[X]]$ for $i\leq s$, $\l_i(X)\in\R[[X]]$ for $i\geq 2s+1$ and $\overline{a}(X)$ denotes the power series whose coefficients are the conjugates of $a(X)$.\\
If $v_i(X)\in\C[[X]]^d$ is an eigenvector associated to $\l_i(X)\notin\R[[X]]$ then $\overline v_i(X)$ is an eigenvector associated to $\overline \l_i(X)$. So we can assume that $A(X)$ has  an orthonormal basis of eigenvectors of the form $v_1$, $\overline v_1$, $v_2$, $\overline v_2$, \ldots, $v_s$, $\overline v_s$, $v_{2s+1}$, \ldots, $v_d$  where $v_{2s+1}$, \ldots, $v_d\in\R[[X]]^d$.
Now let us define
$$u_1=\frac{v_1+\overline v_1}{\sqrt{2}},u_2=i\frac{v_1-\overline v_1}{\sqrt{2}},\ldots, u_{2s-1}=\frac{v_s+\overline v_{s}}{\sqrt{2}}, u_{2s}=i\frac{v_s-\overline v_{s}}{\sqrt{2}}$$
and 
$$u_{2s+1}=v_{2s+1},\ldots, u_d=v_d.$$

The vectors $u_i$ are real and form an orthonormal basis. We have that
$$A(X)u_{2k-1}=A(X)\frac{v_k+\overline v_{k}}{\sqrt{2}}=\frac{1}{\sqrt{2}}(\l_kv_k+\overline\l_k\overline v_{k})=$$
$$=\frac{1}{\sqrt{2}}(\frac{1}{\sqrt{2}}\l_k(u_{2k-1}-iu_{2k})+\frac{1}{\sqrt{2}}\overline \l_k(u_{2k-1}+iu_{2k}))=\frac{\l_k+\overline\l_k}{2}u_{2k-1}+i\frac{\overline\l_k-\l_k}{2}u_{2k}$$
and 
$$A(X)u_{2k}=i\frac{\l_k-\overline\l_k}{2}u_{2k-1}+\frac{\overline\l_k+\l_k}{2}u_{2k}.$$
Therefore in the basis $u_1$, \ldots $u_d$ the matrix has the form \eqref{eq:bigmatrix}.  

   If $A(X)$ is symmetric then the matrix \eqref{eq:bigmatrix} is also symmetric 
   and hence the matrices $C_i(X)$ are symmetric.  Therefore we may assume that $s=0$.        
      \end{proof}


\section{Singular value decomposition}\label{sec:SVD}

Let $A\in Mat_{m,d}  (\C)$.  It is well known (cf. \cite{golub1996}) that 
\begin{align}\label{svd}
A=  U D V^{-1} ,
\end{align}
for some unitary matrices $V \in U_m (\C)$,  
 $U \in U_d (\C)$, and  a 
 (rectangular) diagonal matrix $D$ with real nonnegative coefficients.  The 
diagonal elements of $D$ are the nonnegative square roots of the 
eigenvalues of $A^* A$; they are called \emph{singular values} of $A$. 
If $A$ is real then $V$ and $U$ can be chosen orthogonal. 
The decomposition \eqref{svd} is called \emph{the singular value decomposition} 
(SVD) of $A$.

Let $A\in Mat_{m,d}  (\C [[X]])$.  
Note that   \begin{equation}\label{eq:corres} 
\text{ if $A^*A u = \lambda u$ then $(A A^*)A u = \lambda A u$}. \end{equation} 
Similarly, if $A A^* v = \lambda v$ then $(A^* A )A^* v = \lambda A^*v$.  
Therefore the matrices $A^*A$ and $AA^*$ over the field of formal power series $\C(\!(X)\!)$ have the same nonzero eigenvalues 
 with the same multiplicities.  
  In what follows we suppose $m\le d$.  Then  $P_{A^*A} = Z^{d-m} P_{AA^*}$. 
\begin{theorem} \label{thm:SVD}
Let $A=A(X)\in Mat_{m,d}  (\C [[X]])$, $m\le d$,  and suppose that $\Delta_{A^*A} = X_1^{\alpha_1} \cdots X_n^{\alpha_n} g (X)$ 
with $g(0)\ne 0$.  Then there are unitary matrices $V \in U_m (\C[[X]])$,  
 $U \in U_d (\C[[X]])$ such that 
$$D= V(X)^{-1} A(X) U(X)$$ is (rectangular) diagonal.     

If $A=A(X)\in Mat_{m,d}  (\R [[X]])$ then $U$ and $V$ can be chosen real (that is orthogonal) so that $V(X)^{-1} A(X) U(X)$ is 
block diagonal  as in \eqref{eq:bigmatrix}.  
\end{theorem}

\begin{proof}
We  apply Theorem \ref{thm:normaltheorem} to $A^*A$ and $AA^*$.  Thus there are $U_1 \in U_d(\C[[X]])$,  
 $U_2 \in U_m (\C[[X]])$ such that $D_1= U_1^{-1} A^*A U_1$  and $D_2= U_2^{-1} AA^* U_2$ are diagonal.  If $A(X)$ is real then $A^*A$ and $AA^*$ are symmetric so we may assume by Theorem \ref{thm:realnormaltheorem} that $U_1$ and $U_2$ are orthogonal.

 Set 
 $\hat A= U_2^{-1} A U_1$.  Then 
\begin{align*}
& \hat A^* \hat A = (U_2^{-1} A U_1)^*U_2^{-1} A U _1=  U_1^{-1} A^*A U_1=D_1 \\
&  \hat A \hat A^* = U_2^{-1} A U_1 (U_2^{-1} A U_1)^* =  U_2^{-1}A  A^* U_2 =D_2.  
\end{align*}
Thus by replacing $A$ by $\hat A$ we may assume that both $A^*A$ and $AA^*$ are diagonal and we denote them by $D_1$ and $D_2$ respectively.   

There is a one-to-one correspondence between the nonzero entries of $D_1$ and $D_2$, that is the eigenvalues of $A^*A$ and $AA^*$. 
Let us order these eigenvalues (arbitrarily)
\begin{align}\label{eq:order}
\lambda_1(X), \ldots, \lambda_r(X).
\end{align}
By permuting the canonical bases of $\C[[X]]^m$ and $\C[[X]]^d$ we may assume that the entries on the diagonals of  
$A^*A$ and $AA^*$ appear in the order of  \eqref{eq:order} (with the multiplicities), completed by zeros.  

Since $A$ sends the eigenspace of $\lambda$ of $A^*A$ to the eigenspace of $\lambda$ of $AA^*$,  $A$ is block (rectangular) 
diagonal  in these new bases, with square matrices $A_\lambda$ on the diagonal  corresponding to 
each $\lambda \ne 0$. By symmetry $A^*$ is also block diagonal in these new bases with the square matrices $A_\l^*$ for each $\l\neq 0$. Since $A_\lambda ^*A_\lambda= A_\lambda A_\lambda^*= \lambda \I$, the matrix $A_\lambda$ is normal.  Thus Theorem \ref{thm:normaltheorem} shows that there exist unitary matrices $U'$ and $V'$ such that ${V'}^{-1}AU'$ is diagonal.  
Similarly, by Theorem  \ref{thm:realnormaltheorem} we conclude the real case.  
\end{proof}

\begin{ex}
Consider square matrices of order $1$, that is $d=m=1$, and identify such  a matrix with its entry $a(X)\in\C[[X]]$.  
Then the assumption on the discriminant is always satisfied.    Let us write
$$a(X)=a_1(X)+ia_2(X),\ \ \ a_1(X),a_2(X)\in\R[[X]].$$
A unitary $1\times 1$-matrix corresponds to a series $u(X)=u_1(X)+iu_2(X)$ with $u_1(X)$, $u_2(X)\in\R[[X]]$ such that
$u_1^2+u_2^2=1$. It is not possible in general to find unitary $u$ and $v$ such that $v(X)a(X)u(X)\in\R[[X]]$ and hence 
in Theorem \ref{thm:SVD} we cannot assume that the entries of $D$ are real power series.   Indeed, since all matrices of order $1$ commute it is sufficient to consider the condition $a(X)u(X)\in\R[[X]]$ that  is equivalent to
$$a_1u_2+a_2u_1=0.$$
But if $\gcd(a_1,a_2)=1$, for instance $a_1(X) = X_1, a_2(X) = X_2$, then  $X_1| u_1$ and $X_2|u_2$ 
and hence we see that $u(0) = 0$ that contradicts $u_1^2+u_2^2=1$.  

A similar example in the real case, with $A$ being a block of the form \eqref{blocs} and $a(X)=X_1$, $b(X) = X_2$, shows that we cannot require $D$ to be 
 diagonal in the real case.  
Indeed, in this case the (double) eigenvalue of $A^*A$ is $a^2(X) + b^2(X)$ and it is not the square of an element of $\R[[X]]$.
\end{ex}

\begin{theorem} \label{thm:SVD2}
Suppose in addition to the assumption of Theorem \ref{thm:SVD} that the last nonzero coefficient of 
the characteristic polynomial of $\D_{A^*A}$ is of the form $X_1^{\beta_1} \cdots X_n^{\beta_n} h (X)$ 
with $h(0)\ne 0$.  Then, in the conclusion of Theorem \ref{thm:SVD}, both in the real and the complex case,
 we may require that $V(X)^{-1} A(X) U(X)$ is (rectangular) diagonal with the entries on the diagonal in $\R [[X]]$.

 Moreover the nonzero entries of $V(X)^{-1} A(X) U(X)$  are of the form a monomial times a unit  $X^\alpha a (X) $ (we may additionally require that $a(0)> 0$) and their exponents $\alpha \in \N^n$ are well ordered.

\end{theorem}

\begin{proof}
By the extra assumption each nonzero eigenvalue of $A^*A$ is a monomial times a unit.  The assumption on the discriminant implies the same for all nonzero 
differences of the eigenvalues.  Therefore by \cite[Lemma 4.7]{BM1}, the exponents of these monomials are well ordered. 

In the complex case by Theorem \ref{thm:SVD}  we may assume $A$ diagonal.   Thus it suffices to consider $A$ of order $1$ with the 
entry $a(X)$.  Write $a(X) = a_1(X) +i a_2(X)$ with $a_i(X) \in  \R[[X]]$.  By assumption,  $|a|^2 = \lambda = X^{\beta}  h (X)$ , 
  $h(0)\ne 0$, where $\lambda$ is an eigenvalue of  $A^*A$.   If  $a_1^2(X)+a_2^2(X)$ is a monomial times a unit,  then  the ideal $(a_1(X),a_2(X))$ is generated by a monomial, 
$(a_1(X),a_2(X))= X^\gamma (\tilde a_1(X),\tilde a_2(X))$, $2\gamma = \beta$ 
and $\tilde a_1^2(0) + \tilde a_2^2(0) \ne 0$.   Thus 
$$
a(X) u(X) = X^\gamma  (\tilde a_1^2+ \tilde a_2^2)^{1/2}
$$
with $u(X) = \frac {\tilde a_1 - i \tilde a_2}{(\tilde a_1^2+ \tilde a_2^2)^{1/2}}$.

Let us now show the real case.  It suffices to consider $A$ of the form 
given by \eqref{blocs}.  By assumption,  $a(X)^2+b(X)^2$ is a monomial times a unit and this is possible only 
 if the ideal $(a(X),b(X))$ is generated by a monomial, 
$(a(X),b(X))= X^\gamma (a_0(X), b_0(X))$ and $a_0^2(0) + b_0(0)^2 \ne 0$.  Then 
\begin{align*}
\left[\begin{array}{cc} a & b\\ -b & a\end{array}\right]
\frac 1 {(a_0^2+b_0^2)^{1/2}} \left[\begin{array}{cc} a_0 & - b_0\\ b_0 & a_0\end{array}\right]
= X^\gamma \left[\begin{array}{cc} {(a_0^2+b_0^2)^{1/2}} & 0\\ 0 & {(a_0^2+b_0^2)^{1/2}}\end{array}\right]
\end{align*}
\end{proof}


\section{The case of a Henselian local ring}\label{sec:henselian}
Let $\K=\R$ or $\C$. For every integer $n\in\N$, we consider a subring of $\K[[X_1,\ldots, X_n]]$, denoted by $\K\{\!\!\{X_1,\ldots, X_n\}\!\!\}$. 
For a subrings, we consider the following properties:\\
\begin{equation}\tag{P1}\label{P1} \K\{\!\!\{X_1,\ldots, X_n\}\!\!\} \text{ contains }\K[X_1,\ldots, X_n],\end{equation}
\begin{equation}\tag{P2}\label{P2}\K\{\!\!\{X_1,\ldots, X_n\}\!\!\} \text{ is a Henselian local ring with maximal ideal generated by the }X_i\end{equation}
\begin{equation}\tag{P3}\label{P3}   \K\{\!\!\{X_1,\ldots, X_n\}\!\!\}\cap (X_i)\K[[X_1,\ldots, X_n]]=(X_i)\K\{\!\!\{X\}\!\!\} \text{ for every }i=1, \ldots, n \end{equation}

Let us stress the fact that a ring $\K\{\!\!\{X\}\!\!\}$ satisfying \eqref{P1}, \eqref{P2}, \eqref{P3} is not necessarily Noetherian.

The ring of algebraic $\K\lg X\rg$ or convergent power series $\K\{ X\}$ over $\K$ satisfy \eqref{P1}, \eqref{P2}, \eqref{P3}. In fact any ring satisfying \eqref{P1}, \eqref{P2}, \eqref{P3} has to contain the ring of algebraic power series. 
The ring of germs of $\K$-valued functions defined in a given quasianalytic class (i.e. satisfying (3.1) - (3.6) of \cite{BM}) also satisfies \eqref{P1}, \eqref{P2}, \eqref{P3}.
 \\
Moreover we have the following lemma:

\begin{lemma}
Let $\K\{\!\!\{X\}\!\!\}$ be a ring satisfying \eqref{P1}, \eqref{P2}, \eqref{P3}. Let $f_1$, \ldots, $f_p\in \K\{\!\!\{X\}\!\!\}$ be vanishing at 0, and let $g(Y)\in\K\langle Y_1,\ldots, Y_p\rangle$. Then 
$$g(f_1,\ldots, f_p)\in \K\{\!\!\{X\}\!\!\}.$$
\end{lemma}

\begin{proof}
Since $\K\langle Y\rangle$ is the Henselization of $\K[Y]$, we can write
$$g(Y)=q_0(Y)+\sum_{i=1}^mq_i(Y)g_i(Y)$$
where the $q_i$ are polynomials and the $g_i$ are series of $\K\langle Y\rangle$, $g_i(0)=0$, satisfying the Implicit Function Theorem. That is, for every $i=1,\ldots, m$, there is a polynomial $P_i(Y,T)\in\K[Y,T]$ such that
$$P_i(0,0)=0,\ \frac{\partial P_i}{\partial T}(0,0)\neq 0$$
and $P_i(Y,g_i(Y))=0$. Let us set $f=(f_1,\ldots, f_p)$ and
$$F_i(X,T)=P_i(f(X),T)\in\K\{\!\!\{X\}\!\!\}[T].$$
We have
$$F_i(0,0)=0,\ \frac{\partial F_i}{\partial T}(0,0)\neq 0.$$
Thus $F_i=0$ has a unique solution in $\K[[X]]$ (and even in $\K\{\!\!\{X\}\!\!\}$) vanishing at 0. But  $g_i(f_1,\ldots, f_p)$ is clearly this solution, hence $g_i(f_1,\ldots, f_p)\in \K\{\!\!\{X\}\!\!\}$. Therefore $g(f_1,\ldots, f_p)\in \K\{\!\!\{X\}\!\!\}$.
\end{proof}

We remark that the only tools we use for the proofs of  Theorems \ref{thm:normaltheorem}, 
\ref{thm:realnormaltheorem}, \ref{thm:SVD} are the facts that the ring of formal power series is stable by division by coordinates, the Implicit Function Theorem (via Lemma \ref{lem:SplitMat} which is equivalent to the Henselian property), \and the fact that the ring of formal power series contains the ring of algebraic power series and is stable under composition with algebraic power series (via Lemma \ref{lem:SplitMat} ; see Remark \ref{rem_red}). Therefore, we obtain the following:

\begin{theorem}\label{thm:henselian}
Theorems \ref{thm:normaltheorem} (for $\K=\C$), \ref{thm:realnormaltheorem} (for $\K=\R$), and \ref{thm:SVD} remain valid if we replace $\K[[X]]$ by a ring $\K\{\!\!\{X\}\!\!\}$ satisfying \eqref{P1}, \eqref{P2}, \eqref{P3}.
\end{theorem}


\section{Rectilinearization of the discriminant}\label{sec:rectilinear}

Often the discriminant $\D_A$ does not satisfy the assumption of Theorem  \ref{thm:normaltheorem}, that is it is not a monomial times a unit. Then, in general, 
it is not possible to describe the eigenvalues and eigenvectors of $A$ as (even fractional) power series of $X$.  But this property can be recovered by making the discriminant $\D_A$ normal crossings by means of blowings-up.  This involves a change of the intederminates $X_1, \ldots , X_n$ understood now as variables or local coordinates. Note that in the previous sections all the algebraic operations 
concerned the matrices themselves and not the intederminates $X_1, \ldots , X_n$.  
To stress this difference we will say that we work now in the geometric case

In particular, in the complex case, such a change of local coordinates may affect the other assumption of Theorem  \ref{thm:normaltheorem},  $A$ being normal.
Consider, for instance, the following simple example.  


\begin{ex} (\cite{KP} Example 6.1.)
The eigenvalues of the real symmetric matrix  
\begin{align*}
A= 
\left[\begin{array}{cc} X_1^2 & X_1X_2\\ X_1X_2 & X_2^2 \end{array}\right]
\end{align*}
are $0$ and $X_1^2 + X_2^ 2$ but the eigenvectors of $A$ cannot be chosen as  power series in $X_1, X_2$.  
The discriminant $\Delta_A = (X_1^2 + X_2^ 2)^2$ does not satisfy the assumption of Theorem \ref{thm:normaltheorem}.  

Nevertheless, after a complex change of variables $Y_1=X_1+iX_2, Y_2=X_1 - i X_2$ 
the discriminant $\Delta_A $ becomes a monomial $Y_1^2Y_2^2$. But in these new variables the matrix $A$ is no longer normal, since this change of variables does not commute with the complex conjugation.  
\end{ex}

The above phenomenon does not appear if the change of local coordinates is real.  Therefore, in the normal case we need to work in the real geometric case. We begin by this case.   

Let  $M$ a real manifold belonging to one of the following categories: real analytic, real Nash, or defined in a given quasianalytic class.  In general, the Nash functions are (real or complex)  analytic functions satisfying locally algebraic equations, see e.g \cite{BCR} for the real case. 
Thus $f:(\K^n,0) \to \K$ is the germ of a Nash function if and only if its Taylor series is an algebraic power series.  By a quasianalytic class we mean a class of germs of functions satisfying (3.1) - (3.6) of  \cite{BM}.

We denote by $\O_M$ the sheaf of complex-valued regular 
(in the given category) functions on $M$.  
Let $p\in M$ and let $f\in \O_{M,p}$. We say that $f$ is \emph{normal crossings at $p$} if there is a system of local coordinates at $p$ such that $f$ is equal, in these coordinates, 
to a monomial times a unit. 




\begin{theorem}[Compare Theorem 6.2 of \cite{KP}]\label{blowup}
Let $M$ be a manifold defined in one of the following categories:
\begin{enumerate}
\item[(i)] real analytic; 
\item[(ii)] real Nash;
\item[(iii)] defined in a given quasianalytic class (i.e. satisfying (3.1) - (3.6) of \cite{BM}).
\end{enumerate}
 Let $A\in Mat_{m,d}  (\O_M(M) )$ and let  $K$ be a compact subset of $M$. Then there exist a neighborhood $\Omega$ of $K$ and the composite of a finite sequence of blowings-up  with smooth centers $\pi: U\lgw \Omega$, such that locally on $U$ 

\begin{enumerate}
\item[(a)] if $A$ is a complex normal matrix, then $A\circ\pi$ satisfies the conclusion of Theorem \ref{thm:normaltheorem};
\item[(b)] if $A$ is a real normal matrix, then $A\circ\pi$ satisfies the conclusion of Theorem \ref{thm:realnormaltheorem};
\item[(c)] if $A$ is  not necessarily a square matrix, then $A\circ\pi$  satisfies the conclusion of Theorems \ref{thm:SVD} and \ref{thm:SVD2}.
\end{enumerate}

\end{theorem}

\begin{proof}
 It suffices to apply the resolution of singularities, 
 \cite{Hi} in the Nash case, \cite{BM1} in the analytic case, \cite{BM} in the quasianalytic case, to $f:=\D_A$ in the cases (a) and (b), and to 
 $f:=\D_{A^*A}$  in the case (c). Then $f$ becomes normal crossing, that is locally a monomial times a unit, and we conclude by
  Theorem \ref{thm:henselian}.  
\end{proof}

\begin{rmk}
In the analytic and Nash cases, if $A \in Mat_{m,d}  (\O_M)$ then there exists a globally defined, locally finite  composition of blowings-up with nonsingular centers  $\pi: \widetilde M \to M$, such that 
(a), (b) and (c) are satisfied. Indeed this follows from  \cite{Hi} and \cite[Section 13]{BM2}. 
\end{rmk}

Now we consider the complex geometric case. Let  $M$ a complex manifold belonging either to the complex analytic category, or the complex Nash category. We denote by $\O_M$ the sheaf of complex-valued regular 
(in the given category) functions on $M$.  
Let $p\in M$ and let $f\in \O_{M,p}$. As in the real case, we say that $f$ is \emph{normal crossings at $p$} if there is a system of local complex coordinates at $p$ such that $f$ is equal, in these coordinates, 
to a monomial times a unit. 

 \begin{theorem}\label{blowup:complex}
Let $M$ be a manifold defined in the complex analytic or Nash category.
 Let $A\in Mat_{m,d}  (\O_M)$. Then there exists a locally finite  composition of blowings-up with nonsingular centers  $\pi: \widetilde M \to M$, such that the following holds:\\
 For every $p\in \widetilde M$, there are an open neighborhood of $p$, $\mathcal U_p\subset \widetilde M$, and invertible matrices $V\in Mat_m(\O_{\widetilde M}(U_p))$, $U\in Mat_d(\O_{\widetilde M}(U_p))$, such that $V(A\circ\pi)U$ is rectangular diagonal.
\end{theorem}

\begin{proof}
Indeed in Theorem \ref{thm:SVD}, 
the indeterminates $X$ can be replaced by  complex variables (but here the matrices $U(X)$ and $V(X)$ are no longer unitary since the $X_i$ are complex variables). Therefore the proof of Theorem \ref{blowup:complex} is identical to the proof of Theorem \ref{blowup} cases (a) and (b).
\end{proof}

\section{The global affine case}
Let $U$ be an open set of $\R^n$. We denote by $\O(U)$ the ring of complex valued  Nash functions on $U$, i.e. the ring of real-analytic functions on $U$ that are algebraic over $\C[X_1,\ldots, X_n]$. 
For every point $x\in U$, we denote by $\O(U)_x$ the localization of $\O(U)$ at the maximal ideal defining $x$, i.e. the ideal $\m_x:=(X_1-x_1,\ldots, X_n-x_n)$. The completion of $\O(U)_x$, denoted by $\wdh\O_x$, depends only on $x$ and not on $U$ and is isomorphic to $\C[[X_1,\ldots, X_n]]$. The theorem below can be compared to Theorem 6.2 of \cite{KP}, but note that the latter one is only local.\\

\begin{theorem}\label{global_statement}
Let $U$ be a non-empty simply connected semialgebraic open subset of $\R^n$.
Let the matrix $A\in Mat_d(\O(U))$ be normal and suppose that $\Delta_A$ is 
normal crossings on $U$. Then:
\begin{itemize}
\item[i)]  the eigenvalues of $A$ are in $\O(U)$.  Let us denote by $\l_1$, \ldots, $\l_s$ these distinct eigenvalues;
\item[ii)]  there are  Nash vector sub-bundles $M_i$  of $\O(U)^d$ such that
$$\O(U)^d=M_1\oplus \cdots\oplus M_s ;$$
\item[iii)] for every $u\in M_i$, $Au=\l_iu$.
\end{itemize}
 
\end{theorem}

\begin{proof}
We have that  $P_A\in\O(U)[Z]$. For every $x\in U$ and $Q(Z)\in\O(U)[Z]$ let us denote by $Q_x$  the image of $Q$ in $\wdh \O_x[Z]$. By assumption ${\D_A}_x$ is normal crossings for every $x\in U$.

By Theorem \ref{thm:henselian}, locally at every point of $U$, the eigenvalues of $A$ can be represented by Nash functions, and therefore, since $U$ is simply connected, they are well-defined global functions of $\O(U)$. Let us denote these distinct eigenvalues by $\l_1$,\ldots, $\l_s$ for $s\leq d$.  We set
$$M_i=\Ker(\l_i\I_d-A)\ \ \text{ for } i=1,\ldots, s$$
where $\l_i\I_d-A$ is seen as a morphism defined on $\O(U)^d$. Thus the $M_i$ are sub-$\O(U)$-modules of $\O(U)^d$. 

For an $\O(U)$-module $M$, let us denote by $M_x$ the $\O(U)_x$-module $\O(U)_xM$, and by $\wdh M_x$ the $\wdh \O_x$-module $\wdh\O_x M$.  
By flatness of $\O(U)\lgw \O(U)_x$ and $\O(U)_x\lgw \wdh\O(U)_x$, we have that ${M_i}_x$ is the kernel of $\l_i\I_d-A$ seen as a morphism defined on $\O(U)_x^d$, and $\wdh{M_i}_x$ is the kernel of $\l_i\I_d-A$  seen as a morphism defined on $\wdh \O_x^d$ (see \cite[Theorem 7.6]{Ma}).\\
By Theorem  \ref{thm:normaltheorem}, for every $x\in U$, we have that
$$\wdh{M_1}_x\oplus\cdots\oplus \wdh{M_s}_x=\wdh \O_x^d.$$
Now let us set
$$N=\O(U)^d/(M_1+\cdots+M_s).$$
By assumption for every $x\in U$, we have that $\wdh N_x=0$.
Because $\O(U)$ is Noetherian (see \cite[Th\'eor\`eme 2.1]{Ris}), $\O(U)_x$ is Noetherian. So since $N$ is finitely generated the morphism  $N_x\lgw \wdh N_x$ is injective (see \cite[Theorem 8.11]{Ma}). Therefore $N_x=0$ for every $x\in U$.\\
Thus for every $x\in U$, $\Ann(N)\not\subset \m_x$ where
$$\Ann(N)=\{f\in\O(U)\ \mid\ fN=0\}$$
is the annihilator ideal of $N$. Since the maximal ideals of $\O(U)$ are exactly the ideals $\m_x$ for $x\in U$ (see \cite[Lemma 8.6.3]{BCR}),  $\Ann(N)$ is not a proper ideal of $\O(U)$, i.e. $\Ann(N)=\O(U)$,  and $\O(U)^d=M_1\oplus \cdots\oplus M_s$. \\

For every $x$, we have that ${M_i}_x/\m_x{M_i}_x$ is a $\C$-vector space of dimension $n_{i,x}$ that may depend on $x$ (this vector space is included in the eigenspace of $A(x)$ corresponding to the eigenvalue $\l_i(x)$ - this inclusion may be strict since there may be another $\l_j$ such that $\l_j(x)=\l_i(x)$). So by Nakayama's Lemma every set of $n_{i,x}$ elements of $M_i$ whose images form a $\C$-basis of ${M_i}_x/\m_x{M_i}_x$ is a minimal set of generators of ${M_i}_x$. 
Therefore they make also a minimal set of generators of the $\Frac(\O(U))$-vector space $\Ker(\l_i \I_d-A)$ where $\l_i\I_d-A$ is seen as a morphism defined on $(\Frac (\O(U)))^d$. In particular $n_{i,x}$ is the dimension of   the $\Frac(\O(U))$-vector space $\Ker(\l_i \I_d-A)$ and it is independent of $x$. \\
Now let $u_1$, \ldots, $u_{n_i}\in M_i$ be vectors whose images in ${M_i}_x/\m_x{M_i}_x$ form a basis of ${M_i}_x/\m_x{M_i}_x$. We can write
$$u_j=(u_{j,1},\ldots, u_{j,d})$$ where the $u_{j,k}$ are Nash functions on $U$. So there is a $n_i\times n_i$ minor $\d$ of the matrix $(u_{j,k})$ that does not vanish at $x$, 
and hence there is a neighborhood $V$ of $x$ in $U$ such that for every $\tilde x\in V$, $\d(\tilde x)\neq0$ and the images of $u_1$, \ldots, $u_{n_i}$ form a basis of ${M_i}_{\tilde x}/\m_{\tilde x}{M_i}_{\tilde x}$. We define the morphism of $\O(V)$-modules
$$\Phi :\O(V)^d\longrightarrow M_i(V)$$
 by $\Phi(a_1,\ldots, a_d)=\sum_{j=1}^{n_i} a_ju_j$. Since the $u_j$ generate the stalks ${M_i}_x$ for every $x\in V$, $\Phi_x:\O(V)_x^d\lgw {M_i}_x$ is an isomorphism for every $x\in V$ so $\Phi$ is an isomorphism by \cite[Proposition II.1.1]{Ha}. Hence $M_i$ is a Nash sub-bundle of dimension $n_i$.

\end{proof}


{}

\end{document}